\documentclass[a4paper, twoside, 11pt, english]{article}

\usepackage[T1]{fontenc}
\usepackage[utf8]{inputenc}
\usepackage{etoolbox}
\usepackage{amsmath,amsthm,amssymb,stmaryrd}
\usepackage{mathtools}
\usepackage[ttscale=.875]{libertine}
\usepackage[libertine]{newtxmath}
\usepackage[numbers,sort&compress]{natbib}
\usepackage{comment}
\usepackage{ifthen}
\usepackage{algpseudocode}
\usepackage{algorithm2e}
\usepackage{multicol}
\usepackage{multirow}
\usepackage{fancyhdr, titlesec, url, enumerate, microtype,setspace}
\usepackage{tikz-qtree,varwidth}
\usepackage{youngtab}
\usepackage{ytableau}
\usepackage[english]{babel}
\usepackage{hyperref}
\usepackage[all,knot,poly]{xy}
\usepackage{tikz, tikz-3dplot, pgfplots}
\usepackage[normalem]{ulem}
\usepackage{graphicx} 
\usepackage{float}
\usetikzlibrary{decorations.pathreplacing}
\usetikzlibrary[positioning, patterns]
\usetikzlibrary{arrows.meta}

\usetikzlibrary{calc}
\CompileMatrices
\newcommand{\periodafter}[1]{\ifstrempty{#1}{}{#1.}}
\titleformat{\section}[block]{\scshape\filcenter\LARGE}{\thesection.}{.5em}{}
\titleformat{\subsection}[block]{\bfseries\filcenter\large}{\thesubsection.}{.5em}{\medskip}
\titleformat{\subsubsection}[runin]{\bfseries}{\thesubsubsection.}{.5em}{\periodafter}%{}[.]
\titlespacing{\subsubsection}{0pt}{\topsep}{.5em}

\titleformat{\section}[block]{\scshape\filcenter\Large}{\thesection.}{.5em}{}
\titleformat{\subsection}[runin]{\bfseries}{\thesubsection.}{.5em}{}[.]
\titleformat{\subsubsection}[runin]{\bfseries}{\thesubsubsection.}{.5em}{}[.]
\titlespacing{\subsubsection}{0pt}{10pt}{.5em}

\theoremstyle{plain}
  	\newtheorem{theorem}[subsubsection]{Theorem}
  	\newtheorem{proposition}[subsubsection]{Proposition}
	\newtheorem{lemma}[subsubsection]{Lemma}
  	\newtheorem{corollary}[subsubsection]{Corollary}

\theoremstyle{definition}

	\newtheorem{example}[subsubsection]{Example}
	\newtheorem{remark}[subsubsection]{Remark}

	\makeatletter
\def\@equationname{equation}

\makeatother

\fancypagestyle{plain}{
\fancyhf{}\fancyfoot[LC,RC]{}
\fancyfoot[LE,RO]{$\thepage$}

}

\newcount\hh
\newcount\mm
\mm=\time
\hh=\time
\divide\hh by 60
\divide\mm by 60
\multiply\mm by 60
\mm=-\mm
\advance\mm by \time
\def\hhmm{\number\hh:\ifnum\mm<10{}0\fi\number\mm}

\UseTips
\SelectTips{eu}{11}

\newdir{ >}{{}*!/-10pt/@{>}}
\newdir{ -}{{}*!/-10pt/@{}}
\newdir{> }{{}*!/+10pt/@{>}}

\makeatletter

\xyletcsnamecsname@{dir4{}}{dir{}}
\xydefcsname@{dir4{-}}{\line@ \quadruple@\xydashh@}
\xydefcsname@{dir4{.}}{\point@ \quadruple@\xydashh@}
\xydefcsname@{dir4{~}}{\squiggle@ \quadruple@\xybsqlh@}
\xydefcsname@{dir4{>}}{\Tttip@}
\xydefcsname@{dir4{<}}{\reverseDirection@\Tttip@}

\xydef@\quadruple@#1{%
	\edef\Drop@@{%
		\dimen@=#1\relax
		\dimen@=.5\dimen@
		\A@=-\sinDirection\dimen@
		\B@=\cosDirection\dimen@
		\setboxz@h{%
			\setbox2=\hbox{\kern3\A@\raise3\B@\copy\z@}%
			\dp2=\z@ \ht2=\z@ \wd2=\z@ \box2
			\setbox2=\hbox{\kern\A@\raise\B@\copy\z@}%
			\dp2=\z@ \ht2=\z@ \wd2=\z@ \box2
			\setbox2=\hbox{\kern-\A@\raise-\B@\copy\z@}%
			\dp2=\z@ \ht2=\z@ \wd2=\z@ \box2
			\setbox2=\hbox{\kern-3\A@\raise-3\B@ \noexpand\boxz@}%
			\dp2=\z@ \ht2=\z@ \wd2=\z@ \box2
		}%
		\ht\z@=\z@ \dp\z@=\z@ \wd\z@=\z@ \noexpand\styledboxz@
	}%
}

\xydef@\Tttip@{\kern2pt \vrule height2pt depth2pt width\z@
	\Tttip@@ \kern2pt \egroup
	\U@c=0pt \D@c=0pt \L@c=0pt \R@c=0pt \Edge@c={\circleEdge}%
	\def\Leftness@{.5}\def\Upness@{.5}%
	\def\Drop@@{\styledboxz@}\def\Connect@@{\straight@{\dottedSpread@\jot}}}
	
\xydef@\Tttip@@{%
	\dimen@=.25\dimen@
 	\B@=\cosDirection\dimen@
	\setboxz@h\bgroup\reverseDirection@\line@ \wdz@=\z@ \ht\z@=\z@ \dp\z@=\z@
	{\vDirection@(1,-1)\xydashl@ \xyatipfont\char\DirectionChar}%
	{\vDirection@(1,+1)\xydashl@ \xybtipfont\char\DirectionChar}%
}

\xydef@\ar@form{
	\ifx \space@\next \expandafter\DN@\space{\xyFN@\ar@form}%
	\else\ifx ^\next \DN@ ^{\xyFN@\ar@style}\edef\arvariant@@{\string^}%
	\else\ifx _\next \DN@ _{\xyFN@\ar@style}\edef\arvariant@@{\string_}%
	\else\ifx 0\next \DN@ 0{\xyFN@\ar@style}\def\arvariant@@{0}%
	\else\ifx 1\next \DN@ 1{\xyFN@\ar@style}\def\arvariant@@{1}%
	\else\ifx 2\next \DN@ 2{\xyFN@\ar@style}\def\arvariant@@{2}%
	\else\ifx 3\next \DN@ 3{\xyFN@\ar@style}\def\arvariant@@{3}%
	\else\ifx 4\next \DN@ 4{\xyFN@\ar@style}\def\arvariant@@{4}%
	\else\ifx \bgroup\next \let\next@=\ar@style
	\else\ifx [\next \DN@[##1]{\ar@modifiers{[##1]}}%]
	\else\ifx *\next \DN@ *{\ar@modifiers}%
	\else\addLT@\ifx\next \let\next@=\ar@slide
	\else\ifx /\next \let\next@=\ar@curveslash
	\else\ifx (\next \let\next@=\ar@curveinout %)
	\else\addRQ@\ifx\next \addRQ@\DN@{\ar@curve@}%
	\else\addLQ@\ifx\next \addLQ@\DN@{\xyFN@\ar@curve}%
	\else\addDASH@\ifx\next \addDASH@\DN@{\defarstem@-\xyFN@\ar@}%
	\else\addEQ@\ifx\next \addEQ@\DN@{\def\arvariant@@{2}\defarstem@-\xyFN@\ar@}%
	\else\addDOT@\ifx\next \addDOT@\DN@{\defarstem@.\xyFN@\ar@}%
	\else\ifx :\next \DN@:{\def\arvariant@@{2}\defarstem@.\xyFN@\ar@}%
	\else\ifx ~\next \DN@~{\defarstem@~\xyFN@\ar@}%
	\else\ifx !\next \DN@!{\dasharstem@\xyFN@\ar@}%
	\else\ifx ?\next \DN@?{\ar@upsidedown\xyFN@\ar@}%
	\else \let\next@=\ar@error
	\fi\fi\fi\fi\fi\fi\fi\fi\fi\fi\fi\fi\fi\fi\fi\fi\fi\fi\fi\fi\fi\fi\fi \next@}

\makeatother

\newcommand{\fl}{\rightarrow}

\newcommand{\dfl}{\Rightarrow}

\newcommand{\tfl}{\Rrightarrow}
\newcommand{\qfl}{\xymatrix@1@C=10pt{\ar@4 [r] &}}

\newcommand{\tck}[1]{#1^{\top}}

\DeclareMathOperator{\YoungT}{Yt}

\DeclareMathOperator{\col}{Col}

\DeclareMathOperator{\height}{ht}
\DeclareMathOperator{\rowword}{r}

\renewcommand{\phi}{\varphi}
\renewcommand{\epsilon}{\varepsilon}

\newcommand{\Ar}{\EuScript{A}}

\newcommand{\M}{\mathbf{M}}

\def\catego#1{{\bf{\sf #1}}}

\newcommand{\Cat}{\catego{Cat}}
\newcommand{\Act}{\catego{Act}}

\newcommand{\Colo}{\mathbf{Coln}}
\newcommand{\Std}{\mathbf{Std}}
\newcommand{\inv}{\operatorname{inv}}
\newcommand{\Inv}{\operatorname{Inv}}

\definecolor{cyan}{RGB}{175,238,238} 
\definecolor{Red}{rgb}{0.96,0.17,0.20}
\definecolor{RedD}{rgb}{0.57, 0.0, 0.04}
\definecolor{Green}{rgb}{0.0,1.0,0.0}
\definecolor{GreenL}{rgb}{0.0,1.0,0.0} 
\definecolor{GreenD}{rgb}{0.64,0.76,0.68} 
\definecolor{Yellow}{rgb}{1.0,1.0,0.19}
\definecolor{YellowD}{rgb}{0.93, 0.84, 0.25}
\definecolor{Blue}{rgb}{0.0,1.0,1.0}
\definecolor{BlueD}{rgb}{0.63,0.79,0.95}
\definecolor{vert}{rgb}{0,0.45,0}
\definecolor{bazaar}{rgb}{0.6, 0.47, 0.48}
\definecolor{bronze}{rgb}{0.8, 0.5, 0.2}
\definecolor{darkspringgreen}{rgb}{0.09, 0.45, 0.27}

\def\Ar{\mathcal{A}}

\makeatletter
\def\blfootnote{\xdef\@thefnmark{}\@footnotetext}
\makeatother

\def\Xr{\mathcal{X}}

\begin{document}
\thispagestyle{empty}

\begin{center}

\begin{doublespace}
\begin{huge}
{\scshape The tableaux monoid:\\  a rewriting-theoretic approach}
\end{huge}

\bigskip
\hrule height 1.5pt 
\bigskip

\begin{Large}
{\scshape Nohra Hage }
\end{Large}

\vspace{0.5cm} 

\end{doublespace}

\vspace{1cm} 

\begin{small}\begin{minipage}{14cm}
\noindent\textbf{Abstract -- } 
We study the tableaux monoid from the perspective of rewriting theory.
We first construct a finite quadratic convergent presentation based on column generators and prove that it yields a biautomatic structure.
We then extend this presentation to a finite coherent presentation in the sense of polygraphic rewriting.
These results provide canonical normal forms, yield homological and algorithmic consequences, and produce a higher-dimensional presentation that can be used to describe actions of the tableaux monoid on categories.
\medskip

\smallskip\noindent\textbf{Keywords --} tableaux monoid, convergent presentations, biautomaticity, coherent presentations. 
\medskip

\smallskip\noindent\textbf{M.S.C. 2020 -- Primary:} 20M05, 68Q42, 05E05. \textbf{Secondary:} 05E10, 20M50, 18N30.
\end{minipage}\end{small}

\begin{small}\begin{minipage}{12cm}
\renewcommand{\contentsname}{}
\setcounter{tocdepth}{1}
\tableofcontents
\end{minipage}
\end{small}
\end{center}

\tikzset{every tree node/.style={minimum width=1em,draw,circle},
         blank/.style={draw=none},
         edge from parent/.style=
         {draw,edge from parent path={(\tikzparentnode) -- (\tikzchildnode)}},
         level distance=0.8cm}
\vspace{0.5cm}

\section{Introduction}
\label{S:Introduction}

The combinatorics of Young tableaux has long played a central role in algebraic combinatorics, representation theory, and combinatorial algebra. It appears in a wide variety of settings, including the representation theory of classical groups, crystal bases, and several families of combinatorial monoids and rewriting systems,~\cite{Fulton97,CainGrayMalheiro15,CainGrayMalheiro19,Hage25,Hage2021Super,Hage15,HageMalbos22,HageMalbos17,Hage2022RSK, Hage26}.
Operations obtained by horizontally concatenating rows and then sorting the entries in weakly increasing order within each row had already appeared in related settings: for skew-shaped tableaux in~\cite{Per2016}, and for rectangular
semistandard tableaux in~\cite{ChangDuanFraserLi20}. For semistandard Young tableaux with entries in the ordered
alphabet~$[m]:=\{1,\ldots,m\}$ and height at most~$n$, the product and the resulting monoid were previously defined by Li in~\cite[Section~3]{LiDualCanonical26}. We denote this monoid by~$T_n^m$ and call it the \emph{tableaux monoid}
\footnote{Li~\cite{LiDualCanonical26} and
Daugherty--Gonzalez--Muniz--Ocal--Pan--Torres
\cite{DaughertyGonzalezMunizOcalPanTorres25} use the terminology
``monoid of semistandard Young tableaux.'' We use ``tableaux monoid''
for brevity.}.

Li proved that this monoid is commutative and cancellative~\cite[Lemma~3.2]{LiDualCanonical26}.
Daugherty--Gonzalez--Muniz--Ocal--Pan--Torres subsequently established,
among other properties, that it is torsion-free and generated by its column
tableaux~\cite[Proposition~2.2]
{DaughertyGonzalezMunizOcalPanTorres25}.
They also proved that the associated tableaux algebra is quadratic and has
the Poincar\'e--Birkhoff--Witt (PBW) property, and hence is Koszul
\cite[Theorem~2.13, Corollary~2.14]
{DaughertyGonzalezMunizOcalPanTorres25}.
In the present article, we develop a rewriting-theoretic approach to this
monoid. Our main contributions are a finite quadratic convergent presentation
on the set of columns, a biautomatic structure, and a finite coherent
extension in the sense of polygraphic rewriting. These constructions provide
canonical normal forms, algorithmic and homological consequences, and a
higher-dimensional description of categorical actions.

The starting point is a simple but crucial observation: the tableau product admits a local normalization process. More precisely, multiplication can be described by finitely many quadratic rules acting on adjacent columns. This perspective makes it possible to study the monoid both algorithmically and homotopically.

Our first main result is a finite quadratic convergent presentation of~$T_n^m$ on the set of columns; see Theorem~\ref{thm:ColumnPresentation}. The normal forms for this presentation are exactly the semistandard sequences of columns, namely the column decompositions of tableaux. As a consequence, the tableaux monoid admits a canonical language of normal forms, and it is of type~$\mathrm{FP}_\infty$ and~$\mathrm{FDT}_\infty$.

Our second main result is the construction of a biautomatic structure for~$T_n^m$; see Theorem~\ref{thm:biautomatic}. Using the column normal forms, we describe right multiplication by a generator by means of a finite-state transducer based on a single right-to-left normalization pass. Since the monoid is commutative, the same construction also governs left multiplication.  As a consequence, we deduce a quadratic-time solution to the word problem. Using the definition of the star product, we also establish decidability of divisibility and prove that all Green relations on~$T_n^m$ are trivial. In this way, the tableaux monoid joins the broader family of well-behaved combinatorial monoids, such as plactic-like monoids~\cite{CainGrayMalheiro15,CainGrayMalheiro19,CainGrayMalheiro15b}, while retaining a comparatively simple normalization mechanism.

Our third main result is a finite coherent extension of the column presentation; see Theorem~\ref{T:ColumnCoherentPresentation}. In the language of polygraphic rewriting, this provides a higher-dimensional presentation of the tableaux monoid in which the critical branchings of the convergent column presentation are filled by generating $3$-cells. Comparable coherent presentations have already been constructed for Artin monoids~\cite{GaussentGuiraudMalbos15}, monoids admitting a right-noetherian Garside family~\cite{CurienDuricGuiraud23}, plactic monoids of several types~\cite{HageMalbos17, Meha2021}, Chinese monoids~\cite{HageMalbos22}, and stylic monoids~\cite{hage2026coherent}. The coherent column presentation can be extended to a polygraphic resolution of~$T_n^m$ by the higher-dimensional
normalization methods of~\cite{GuiraudMalbos12advances}. Coherent presentations also provide a convenient framework for describing actions of the monoid on categories~\cite{GaussentGuiraudMalbos15}.

The paper is organized as follows. Section~\ref{S:ColumnPresentationsTableauxMonoid} reviews the necessary background on string rewriting and establishes the finite quadratic convergent column presentation. 
Section~\ref{S:BiautomaticStructuresTableauxMonoid} establishes biautomaticity, derives a quadratic-time solution to the word problem, and studies divisibility and Green's relations. Section~\ref{S:ColumnCoherentPresentationsTableauxMonoid} constructs a finite coherent extension, derives a rewriting proof of Koszulity, examines the size of the presentation, and concludes with a categorical interpretation of the coherent data.

\section{Finite convergent presentations of the tableaux monoid}
\label{S:ColumnPresentationsTableauxMonoid}

In this section we develop finite convergent presentations of the tableaux monoid.
After recalling the basic notions of string rewriting, we introduce a finite
quadratic column presentation and prove its convergence.
This yields unique normal forms and implies that the monoid is of type~$\mathrm{FP}_\infty$ and $\mathrm{FDT}_\infty$.

\subsection{Convergent presentations of monoids}
\label{SS:ConvergentPresentations}
We begin by recalling the basic definitions and properties of string rewriting systems.
General references for rewriting theory include~\cite{BookOtto93}. 
Let~$X$ be a set.
We denote by $X^\ast$ the free monoid generated by $X$, whose elements are called \emph{words},
with multiplication given by concatenation and neutral element the empty word~$\varepsilon$.
For a word $u = x_1 \cdots x_k$ with $x_i$ in~$X$, its length is denoted by $|u| = k$.
A \emph{string rewriting system} on~$X$ is a set $R \subseteq X^\ast \times X^\ast$.
An element $\beta = (u,v)$ of~$R$ is called a \emph{rewriting rule}, written~$\beta : u \fl v$,
where $u$ is the \emph{source} and $v$ the \emph{target} of the rule.
We denote these by~$s(\beta)=u$ and~$t(\beta)=v$.
A \emph{one-step reduction} induced by $\beta$ is a transformation~$wuw' \;\fl\; wvw'$
for arbitrary words $w,w'$  in~$X^\ast$, and is denoted by $w\beta w'$.
The set of all one-step reductions generates a binary relation on~$X^\ast$, written~$\fl_R$.
Its reflexive and transitive closure is denoted by~$\fl_R^\ast$ and called the \emph{rewriting relation}.
A word $u$ is said to be \emph{irreducible} if no rule of $R$ applies to it.
A \emph{normal form} of a word~$u$ is an irreducible word~$v$ such that $u \fl_R^\ast v$. 
The congruence on~$X^\ast$ generated by $R$ is the equivalence relation
obtained as the reflexive, symmetric, and transitive closure of $\fl_R$,
and is denoted by $\approx_R$.
The monoid presented by $R$ is the quotient~$X^\ast / \!\approx_R$.

The system $R$ is said to be \emph{terminating} if there is no infinite rewriting sequence,
and \emph{normalizing} if every word admits at least one normal form.
The system $R$ is said to be \emph{quadratic} if every rule has a source and a target of length $2$.
The system~$R$ is \emph{reduced} if, for every rule
$\beta:u\fl v$ in~$R$, the word~$v$ is irreducible with
respect to~$R$, and~$u$ is irreducible with respect to
$R\setminus\{\beta\}$.
If~$R$ is quadratic and every reducible word of length~$2$ is the source of at most one rule, as for the column
presentation below, then a word of length~$3$ admits at most two distinct one-step reductions, one in each adjacent position.

A \emph{branching} of a rewriting system $R$ is a pair $(f,g)$ of rewriting paths with the same source.
If both $f$ and $g$ consist of a single rewriting step, the branching is said to be \emph{local}.
Local branchings are classified as \emph{aspherical}, when $f=g$,  \emph{Peiffer}, when the two reductions act on disjoint subwords, and  \emph{overlapping}, in all other cases.
An overlapping local branching is called \emph{critical} if it is minimal with respect to contextual inclusion,
that is, it cannot be obtained by inserting a smaller branching into a larger word.
A branching~$(f,g)$ is said to be \emph{confluent} if there exist rewriting paths~$f'$ and~$g'$ starting at the targets of~$f$ and~$g$, respectively, and ending at a common word.
The rewriting system $R$ is \emph{confluent} if all branchings are confluent,
and \emph{convergent} if it is both confluent and terminating.
In that case, every word admits a unique normal form. A fundamental result linking termination and confluence is due to Newman.
It states that, for a terminating rewriting system, confluence is equivalent
to local confluence,~\cite{Newman42,BookOtto93}.

\subsection{Tableaux monoids}
\label{SS:Tableauxmonoids}
Let~$m$ and $n$ be in~$\mathbb{N}$ with~$m\ge n\ge1$.
A \emph{partition} of height at most~$n$ is a weakly decreasing
sequence of nonnegative integers
\[
\lambda=(\lambda_1,\lambda_2,\ldots,\lambda_n)
\qquad\text{with}\qquad
\lambda_1\ge \lambda_2\ge \cdots \ge \lambda_n\ge 0.
\]
We allow the \emph{empty partition}~$\varnothing=(0,\ldots,0)$.
The \emph{height} of~$\lambda$ is the number of nonzero parts.
A partition~$\lambda$ is identified with its \emph{Ferrers diagram}
(in English notation), namely the left-justified array of boxes having~$\lambda_i$ boxes in the $i$-th row.
A \emph{semistandard Young tableau}, or \emph{tableau} for short,  of shape $\lambda$ is a filling of the boxes of the Ferrers diagram of $\lambda$ with entries of~$[m]$ such that the rows are weakly increasing left-to-right and the columns are strictly increasing top-to-bottom. 
For the empty partition~$\varnothing$, there is a unique tableau,
called the \emph{empty tableau}.
Denote by~$\YoungT_m^n$ the set of all tableaux of all possible partition shapes with at most~$n$ rows and entries in~$[m]$, and by~$\YoungT_m(\lambda)$ the set of all tableaux of fixed shape partition~$\lambda$ with entries in~$[m]$.

Let $T$ be in~$\YoungT_m(\lambda)$ and $T'$ be
in~$\YoungT_m(\mu)$.
The \emph{star product} $T \star T'$ is the tableau in
$\YoungT_m(\lambda+\mu)$ obtained by horizontally concatenating
the rows of $T$ and $T'$ and then, in each row independently,
sorting the resulting entries in weakly increasing order
from left to right.
This product is well defined, associative, and commutative,
see~\cite{LiDualCanonical26} and
\cite[Proposition~2.2]{DaughertyGonzalezMunizOcalPanTorres25}.
For instance, consider the following tableaux in~$\YoungT_3^2$:
\[
T =
\ytableausetup{mathmode, boxsize=1em}
\begin{ytableau}
1&2\\
3
\end{ytableau},
\qquad
T' =
\ytableausetup{mathmode, boxsize=1em}
\begin{ytableau}
1\\
2
\end{ytableau}.
\]
Then
\[
T\star T' =
\ytableausetup{mathmode, boxsize=1em}
\begin{ytableau}
1&1&2\\
2&3
\end{ytableau}.
\]
The empty tableau~$\emptyset$ satisfies
$T\star\emptyset=\emptyset\star T=T$ for every~$T$
in~$\YoungT_m^n$.
Thus~$(\YoungT_m^n,\star,\emptyset)$ is a commutative monoid,
which we denote by~$T_n^m$ and call the
\emph{tableaux monoid}.

Set~$\YoungT_{m,+}^n=\YoungT_m^n\setminus\{\emptyset\}$.
We denote by~$\Std$ the \emph{standard rewriting system} whose generating
set is~$\YoungT_{m,+}^n$ and whose rules are
\[
TT'\;\fl\;T\star T',
\qquad
T,T'\in\YoungT_{m,+}^n.
\]
Here the empty word~$\varepsilon$ represents the empty tableau and hence the
unit of the monoid.
Every rewriting step decreases word length by~$1$, so~$\Std$ is terminating.
Its overlapping critical branchings have source~$TT'T''$, and the two
resulting reductions are completed using~$(T\star T')\star T''=T\star(T'\star T'')$.
Thus~$\Std$ is confluent. Its irreducible words are precisely~$\varepsilon$ and the one-letter words~$T$ with
$T\in\YoungT_{m,+}^n$, which are naturally identified with all tableaux,
including the empty tableau. Under this identification, multiplication in
the presented monoid is~$\star$. Therefore~$\Std$ is a convergent
presentation of~$T_n^m$, called the \emph{standard presentation}.

\subsection{The column presentation}
\label{SS:ColumnPresentation}
For $1 \le k \le n$, let $\YoungT_m(1^k)$ denote the set of 
tableaux of shape~$(1^k)$ with entries in $[m]$.
Set
\[
\col_n(m) \;=\; \bigcup_{k=1}^n \YoungT_m(1^k).
\]
Elements of $\col_n(m)$ are called \emph{columns}.
The \emph{height} of a column~$A$, denoted~$\height(A)$, is the number of boxes in~$A$.
Each column $A$ in~$\col_n(m)$ of height $k$ is uniquely determined by its \emph{column reading}, that is, the strictly increasing $k$-tuple
$(a_1,\dots,a_k)$ in~$[m]^k$ obtained by listing its entries from top to bottom.
Conversely, every strictly increasing $k$-tuple in $[m]^k$ determines a unique column tableau of shape $(1^k)$.
We therefore identify columns with their column readings whenever convenient.
For $1\le i\le k$, we write $A(i)=a_i$, where rows are indexed from top to
bottom so that $A(1)$ denotes the top entry and $A(k)$ the bottom entry.
Given columns $A$ and~$B$ in~$\col_n(m)$, we say that the pair $AB$ is
\emph{semistandard} if the tableau obtained by placing $B$ immediately to the
right of $A$ is semistandard.
If $A$ has height $k$ and $B$ has height $l$, this condition is equivalent to
$l\le k$ and $A(i)\le B(i)$ for all $1\le i\le l$.
If this condition is not satisfied, the pair $AB$ is said to
\emph{require normalization}.

\begin{lemma}
\label{L:TwocolumnsLemmata}
Let $A$ and~$B$ be in~$\col_n(m)$.
Then the tableau~$A\star B$ always consists of exactly two columns.
\end{lemma}

\begin{proof}
By definition, the tableau $A \star B$ is obtained by, in each row $i$,
concatenating the entries $A(i)$ and $B(i)$ when both are defined and then
sorting them in weakly increasing order.
Thus each row of $A \star B$ contains two entries whenever both $A(i)$ and
$B(i)$ are defined, and a single entry otherwise.
Reassembling the rows columnwise yields a tableau with exactly two columns,
whose left (resp.~right) column consists of the smaller (resp.~larger)
entry in each row.
\end{proof}

We consider words over~$\col_n(m)$ as formal horizontal concatenations of columns.
Define the \emph{column decomposition map}
\[
\col : \YoungT_m^n \longrightarrow (\col_n(m))^\ast
\]
by reading a tableau from left to right as a sequence of columns. In particular, we have~$\col(\emptyset)=\varepsilon$.
The map $\col$ is injective, and its image consists precisely of the
words $C_1\cdots C_k$ such that each adjacent pair $C_iC_{i+1}$ is
semistandard.
Indeed, every tableau decomposes uniquely into its sequence of columns.
Conversely, every word~$C_1\cdots C_k$ such that each adjacent pair
$C_iC_{i+1}$ is semistandard assembles into a unique tableau.

We now define the rewriting system $\Colo$.
Its generating set is~$\Colo_1=\col_n(m)$.
Given columns~$A,B$ in~$\col_n(m)$ such that the pair~$AB$ is not
semistandard, let~$\col(A\star B)=A'B'$.
To this pair we associate the rewriting rule
\[
\alpha_{A,B}: AB \;\fl\; A'B'.
\]
Accordingly, the set of rewriting rules is
\[
\Colo_2=
\{\alpha_{A,B}\mid A,B\in\col_n(m)\text{ and }AB\text{ is not semistandard}\}.
\]

\begin{theorem}
\label{thm:ColumnPresentation}
The column rewriting system~$\Colo$ is a finite, reduced, quadratic, and convergent
presentation of the tableaux monoid~$T_n^m$, called the \emph{column presentation}.
\end{theorem}

\begin{proof}
Since the set of columns~$\col_n(m)$ is finite and, by construction, each rewriting rule~$\alpha_{A,B}$ has both source and target of length~$2$, the rewriting system~$\Colo$ is finite and quadratic.
Moreover,~$\Colo$ is reduced. For every rule
$\alpha_{A,B}:AB\fl A'B'$, the target
$A'B'=\col(A\star B)$ is semistandard and therefore irreducible.
Every rule has a source of length~$2$, and~$\alpha_{A,B}$
is the unique rule with source~$AB$.
Hence~$AB$ is irreducible with respect to
$\Colo_2\setminus\{\alpha_{A,B}\}$.

Let us show that~$\Colo$ is terminating.
Set
\[
\widehat{[m]}=[m]\sqcup\{\infty\},
\qquad
1<2<\cdots<m<\infty .
\]
For a column~$A$ of height~$k$, define its \emph{padded entries} by
\[
\widehat{A}(i)=
\begin{cases}
A(i), & \text{if } 1\le i\le k,\\
\infty, & \text{if } k<i\le n.
\end{cases}
\]
Then a pair of columns~$AB$ is semistandard if and only if~$\widehat{A}(i)\le \widehat{B}(i)$, for all  $1\le i\le n$.
Now let~$w=C_1\cdots C_p$ be in~$\col_n(m)^\ast$.
For each row index~$1\le i\le n$, define the \emph{padded row word}
\[
\widehat{\rowword}_i(w)
=
\widehat{C_1}(i)\cdots \widehat{C_p}(i)
\in \widehat{[m]}^\ast.
\]
For a word~$x=x_1\cdots x_p$ in~$\widehat{[m]}^\ast$, define
\[
\inv(x)
=
\#\{(r,s)\mid 1\le r<s\le p,\ x_r>x_s\}.
\]
For~$w$ in~$\col_n(m)^\ast$, set
\[
\Inv(w)
=
\sum_{i=1}^{n}\inv\bigl(\widehat{\rowword}_i(w)\bigr).
\]
Since~$\Inv(w)\in\mathbb N$, it suffices to show that every rewriting step strictly decreases~$\Inv$.
Let~$\alpha_{A,B}$ be a rewriting rule of~$\Colo$, so that~$\col(A\star B)=A'B'$.
By construction, for every~$1\le i\le n$, we have
\[
\widehat{A'}(i)=\min\{\widehat{A}(i),\widehat{B}(i)\}\quad\text{ and }\quad \widehat{B'}(i)=\max\{\widehat{A}(i),\widehat{B}(i)\}.
\]
Thus, in each padded row word, the rewrite replaces the adjacent factor~$\widehat{A}(i)\widehat{B}(i)$
by its weakly increasing rearrangement.
If~$\widehat{A}(i)\le \widehat{B}(i)$, then the $i$-th padded row word is unchanged.
If~$\widehat{A}(i)>\widehat{B}(i)$, then the adjacent inversion~$\widehat{A}(i)\widehat{B}(i)$
is removed. Since the two letters are adjacent, their inversion relations with all other letters in the row remain unchanged, and therefore the inversion number of that row decreases by exactly~$1$.
Because~$AB$ is not semistandard, there exists at least one row~$i$ such that~$\widehat{A}(i)>\widehat{B}(i)$.
It follows that
\[
\Inv(UA'B'V)<\Inv(UABV)
\]
for all words~$U,V$ in~$\col_n(m)^\ast$.
Hence every rewriting step strictly decreases the natural number~$\Inv$, so no infinite rewriting sequence can exist.
Therefore~$\Colo$ is terminating.

Let us now prove that~$\Colo$ is confluent.
Since~$\Colo$ is terminating, it is enough by Newman's Lemma to establish local confluence.
Aspherical and Peiffer local branchings are trivially confluent, so it remains to consider only overlapping local branchings.
Since~$\Colo$ is quadratic and has exactly one rule for each non-semistandard pair, every overlapping local branching is obtained by adding left and right contexts to a critical branching with source~$ABC$, where~$A,B,C\in\col_n(m)$ and both~$AB$ and~$BC$ are non-semistandard.
The two one-step reductions are the column words~$\col(A\star B)C$ and~$A\col(B\star C)$, which represent the tableaux~$(A\star B)\star C$ and~$A\star(B\star C)$, respectively.
By associativity of the product~$\star$, these two tableaux are equal.
By termination, both one-step reducts admit irreducible reducts.
Every rewriting step preserves the star product of the columns,
and every irreducible column word is the column decomposition
of its tableau product.
Therefore both irreducible reducts are equal to~$\col(A\star B\star C)$, which is a common irreducible reduct.
Hence every local branching is confluent.
Since~$\Colo$ is terminating, Newman's Lemma implies that~$\Colo$ is confluent.

Let us finally show that~$\Std$ and~$\Colo$ present the same monoid.
Every nonempty tableau~$T\in\YoungT_{m,+}^n$ admits a unique column decomposition~$\col(T)=C_1\cdots C_k$, with~$k\ge1$ and~$C_i\in\col_n(m)$. The empty tableau is represented by~$\varepsilon$ in both
presentations.
Since each~$C_i$ is itself a tableau, repeated application of the defining rules of~$\Std$ yields
\[
C_1\cdots C_k \;\fl_\Std^\ast\; C_1\star\cdots\star C_k = T.
\]
Thus, in the monoid presented by~$\Std$, every tableau generator is equal to its corresponding column word.
Under this identification, a defining relation of~$\Std$ becomes
\[
\col(T)\col(T') \;\approx_\Std\; \col(T\star T').
\]
This relation follows from the rules of~$\Colo$, since successive normalization of adjacent non-semistandard pairs transforms the word~$\col(T)\col(T')$ into the unique semistandard column word representing the same tableau product, namely~$\col(T\star T')$.
Conversely, let $A,B\in\col_n(m)$ and write~$\col(A\star B)=A'B'$.
In the monoid presented by~$\Std$, one has~$AB = A\star B = A'B'$.
Hence every defining relation of~$\Colo$ holds in the monoid presented by~$\Std$.
Therefore, the presentations~$\Std$ and~$\Colo$ define the same congruence on column words and present the same monoid.
Since~$\Std$ presents the tableaux monoid~$T_n^m$, it follows that~$\Colo$ also presents~$T_n^m$.
\end{proof}

\subsubsection{PBW basis and normal forms}

Since~$\Colo$ is convergent, its irreducible column words are
precisely the words~$C_1\cdots C_p$ for which every adjacent
pair~$C_iC_{i+1}$ is semistandard.
Column decomposition identifies these words bijectively with  tableaux, with the empty word corresponding
to the empty tableau.
Order the columns lexicographically by their padded entry vectors,
using~$1<\cdots<m<\infty$, and equip column words with the induced
degree-lexicographic order.
For every rule~$AB\fl A'B'$, one has~$\widehat{A'}(i)\le\widehat A(i)$, for every~$1\le i\le n$, with strict inequality for at least one~$i$.
Thus~$A'<A$ lexicographically and~$A'B'<AB$ in the
degree-lexicographic order.
For any field~$\Bbbk$, convergence of~$\Colo$ and compatibility
with this monomial order imply that the defining
binomials~$AB-A'B'$ form a quadratic noncommutative
Gr\"obner basis.
Consequently, the irreducible monomials form a PBW
$\Bbbk$-basis of~$\Bbbk[T_n^m]$ associated with the column
presentation, see~\cite[Proposition~3.5.4]{GuiraudHoffbeckMalbos19}.
This gives an explicit normal-form realization of the PBW
property established in the proof of
\cite[Corollary~2.14]{DaughertyGonzalezMunizOcalPanTorres25}.

\subsubsection{Homological finiteness properties}

A monoid~$\M$ is of type left-$\mathrm{FP}_n$ if the trivial left
$\mathbb Z\M$-module~$\mathbb Z$ admits a projective resolution whose
modules are finitely generated through dimension~$n$. It is of type
left-$\mathrm{FP}_\infty$ if this holds in every dimension. The analogous
right-handed condition is defined using right modules.
A monoid is of type~$\mathrm{FDT}_3$ if it admits a finite
coherent presentation (see Subsection~\ref{SS:CoherentPresentations}), and of type~$\mathrm{FDT}_\infty$
if it admits a polygraphic resolution with finitely many
generating cells in each dimension.
Squier proved that a monoid admitting a finite convergent presentation is
of type left-$\mathrm{FP}_3$ and has finite derivation type~$\mathrm{FDT}_3$
\cite{Squier87,Squier94}. Higher-dimensional refinements of Squier's
construction associate with a finite convergent presentation a polygraphic
resolution having finitely many generating cells in every dimension.
Consequently, such a monoid is of type~$\mathrm{FDT}_\infty$ and
left-$\mathrm{FP}_\infty$,  see
\cite{GuiraudMalbos09, GuiraudMalbos12advances}.
For a commutative monoid the corresponding left-module and right-module
statements agree, and we abbreviate this common property to~$\mathrm{FP}_\infty$.

Since~$\Colo$ is finite and convergent by
Theorem~\ref{thm:ColumnPresentation}, we obtain:

\begin{corollary}
\label{Corollary:FP-FDT}
The tableaux monoid~$T_n^m$ is of type~$\mathrm{FDT}_\infty$ and of type~$\mathrm{FP}_\infty$.
\end{corollary}

\section{Biautomatic structures for the tableaux monoid}
\label{S:BiautomaticStructuresTableauxMonoid}

In this section we prove that the tableaux monoid admits a
\emph{biautomatic structure}.
The proof is based on the finite convergent column presentation.
The locality and length-preserving nature of the column rewriting
rules make it possible to simulate left and right multiplication by
generators using finite-state transducers.

\subsection{Automatic and biautomatic structures}
\label{SS:Automaticity}

We briefly recall the notions of automatic and biautomatic monoids
used in this article. Standard references include
\cite{DuncanRobertsonRuskuc99,CainGrayMalheiro19,CainGrayMalheiro15b}.

Let $\Ar$ be a finite alphabet.
A \emph{finite-state transducer} is a tuple
\[
\mathcal T = (Q,\Ar,\Ar,\delta,q_0,F),
\]
where $Q$ is a finite set of states, $q_0\in Q$ is the initial state,
$F\subseteq Q$ is the set of accepting states, and
\[
\delta \subseteq Q\times (\Ar\cup\{\varepsilon\})
\times (\Ar\cup\{\varepsilon\})\times Q
\]
is a finite set of transitions.
We allow $\varepsilon$-transitions, so a transducer may in particular perform a final output step after the input has been exhausted.
A transition $(p,a,b,q)$, written
$p \xrightarrow{a\mid b} q$,
means that the machine may read the symbol $a$, output the symbol $b$,
and move from state $p$ to state $q$.
A pair of words $(u,v)$ in~$\Ar^*\times\Ar^*$ is accepted if
there exists a path from $q_0$ to some state in $F$
whose sequence of input labels spells $u$ and whose output labels spell $v$.
The set of all such pairs is a binary relation
$R\subseteq\Ar^*\times\Ar^*$,
called a \emph{rational relation}.
Thus a rational relation is recognized by a finite-state device that reads
the two tapes asynchronously, allowing~$\varepsilon$ on either tape.

Automaticity is defined in terms of regularity of certain padded relations.
Let $\$$ be a symbol not in~$\Ar$.
Given two words $u,v$ in~$\Ar^*$,
their right-padded version $\delta_R(u,v)$
is obtained by appending $\$$ symbols to the shorter word
so that both words have the same length,
and then reading them synchronously as a word over
$(\Ar\cup\{\$\})\times(\Ar\cup\{\$\})$.
Similarly, the left padding map $\delta_L$
pads the shorter word on the left.
These padding maps allow one to reduce questions about
binary relations on $\Ar^*$ to questions about regular languages.

Let $\M$ be a monoid generated by~$\Ar$,
and let $L\subseteq\Ar^*$ be a regular language
that maps onto~$\M$.
The language $L$ is intended to serve as a set of normal forms.
For words~$u,v$ in~$\Ar^*$, we write~$u=_\M v$ if they
represent the same element of~$\M$.
For $a$ in~$\Ar\cup\{\varepsilon\}$, define
\[
L_a=\{(u,v)\in L\times L \mid ua=_\M v\},
\qquad
{}_aL=\{(u,v)\in L\times L \mid au=_\M v\}.
\]
Thus $L_a$ encodes right multiplication by~$a$,
while ${}_aL$ encodes left multiplication.
The pair~$(\Ar,L)$ is called an \emph{automatic structure} if, for every
$a\in\Ar\cup\{\varepsilon\}$, the padded relation~$L_a\delta_R$ is regular.
It is called a \emph{biautomatic structure} if, for every
$a\in\Ar\cup\{\varepsilon\}$, all four relations~$L_a\delta_R$,~${}_aL\delta_R$,~$
L_a\delta_L$ and~${}_aL\delta_L$ are regular.

In practice, one constructs a rational relation describing
multiplication by a generator and then proves that it has bounded
length difference.
The following standard result, obtained by combining
\cite[Corollary~2.5]{FS93} and \cite[Proposition~4]{HT05},
allows one to deduce regularity:

\begin{proposition}
\label{prop:rational-to-regular}
Let $R\subseteq \Ar^*\times\Ar^*$ be rational.
If there exists $k\ge0$ such that
\[
\bigl||u|-|v|\bigr| \le k
\quad\text{for all } (u,v)\in R,
\]
then $R\delta_R$ and $R\delta_L$ are regular.
\end{proposition}

This criterion will be used below.
We construct finite-state transducers describing multiplication
in the tableaux monoid and prove that the difference in lengths
between input and output normal forms is uniformly bounded.
It follows that the tableaux monoid is biautomatic.

\subsection{Normal forms}
Let
\[
K=\{A_1\cdots A_k\in \col_n(m)^*
\mid A_iA_{i+1}\text{ is semistandard for all }i\}.
\]
By convergence of $\Colo$,
the language $K$ is precisely the set of normal forms.

\begin{lemma}
The language $K$ is regular.
\end{lemma}

\begin{proof}
To verify semistandardness, it suffices to check adjacent pairs.
An automaton reading a word in~$\col_n(m)^*$ needs only to store the
previous column.
Since $\col_n(m)$ is finite, this requires finitely many states.
\end{proof}

\subsection{Right multiplication}

Fix~$C$ in~$\col_n(m)$ and a normal form~$u=A_1\cdots A_k$ in~$K$.
Since~$\Colo$ is length preserving, the~$\Colo$-normal form of~$uC$ has
length~$|u|+1$.

We compute the normal form of $uC$ by a single right-to-left pass,
carrying a column through the word.
Define inductively a sequence of carried columns by setting $D_k=C$.
For $i=k,k-1,\dots,1$, normalize the adjacent pair $A_iD_i$ by computing~$\col(A_i\star D_i)=D_{i-1}B_i$.
Thus at step $i$ the column~$B_i$ is produced on the right, while
$D_{i-1}$ is carried to the next step.

We maintain the following invariant.
After processing positions $k,k-1,\dots,i+1$, the suffix
$B_{i+1}\cdots B_k$ is $\Colo$-irreducible and
\[
A_{i+1}\cdots A_k C \; =_{T_n^m}\; D_i\,B_{i+1}\cdots B_k.
\]
The invariant holds initially since~$D_k=C$.
For~$k\ge1$, the first step produces the irreducible suffix~$B_k$.
For~$1\le i<k$, assume that the invariant holds before step~$i$.
We prove that after step $i$ the suffix
$B_iB_{i+1}\cdots B_k$ remains $\Colo$-irreducible.
Since $B_{i+1}\cdots B_k$ is already irreducible, it suffices to show
that the pair $B_iB_{i+1}$ is semistandard.
Because~$u$ lies in $K$, the pair $A_iA_{i+1}$ is semistandard.
Hence $\height(A_{i+1})\le \height(A_i)$ and
$A_i(r)\le A_{i+1}(r)$ for all~$r\le \height(A_{i+1})$.
Moreover, from~$\col(A_{i+1}\star D_{i+1})=D_iB_{i+1}$,
the columns $D_i$ and $B_{i+1}$ are obtained rowwise by taking minima
and maxima:
\[
D_i(r)=\min\{A_{i+1}(r),D_{i+1}(r)\},
\qquad
B_{i+1}(r)=\max\{A_{i+1}(r),D_{i+1}(r)\},
\]
whenever these entries are defined.
In particular,~$D_i(r)\le A_{i+1}(r)\le B_{i+1}(r)$.
Since $B_i$ is the right column of $\col(A_i\star D_i)$, its entries are
given by~$B_i(r)=\max\{A_i(r),D_i(r)\}$.
It follows that for every row~$r\le \height(B_{i+1})$,
\[
B_i(r)=\max\{A_i(r),D_i(r)\}\le A_{i+1}(r)\le B_{i+1}(r).
\]
Finally,
\[
\height(B_{i+1})
\le \height(A_{i+1})
\le \height(A_i)
\quad\text{and}\quad
\height(B_{i+1})\le \height(D_i),
\]
so that
\[
\height(B_{i+1})
\le \min(\height(A_i),\height(D_i))
= \height(B_i).
\]
Thus $B_i(r)\le B_{i+1}(r)$ whenever both entries are defined and
$\height(B_{i+1})\le \height(B_i)$, showing that the pair~$B_iB_{i+1}$
is semistandard.
This proves that the suffix remains $\Colo$-irreducible and that the
invariant is preserved.

After completing the pass, we obtain the word~$v=D_0B_1\cdots B_k$.
The invariant gives~$uC=_{T_n^m}v$ and shows that~$B_1\cdots B_k$ is~$\Colo$-irreducible.
If~$k\ge1$, the pair~$D_0B_1=\col(A_1\star D_1)$ is
semistandard by construction, and if~$k=0$, one has~$v=C$.
Thus~$v$ is~$\Colo$-irreducible and, by convergence,
is the unique normal form of~$uC$.

To formalize this process, we use the standard reversal trick.
For a word~$u=A_1\cdots A_k$, write
$u^{\mathrm{rev}}=A_k\cdots A_1$, with
$\varepsilon^{\mathrm{rev}}=\varepsilon$.
If $R\subseteq \col_n(m)^*\times\col_n(m)^*$, define
\[
R^{\mathrm{rev}}
=
\{(u^{\mathrm{rev}},v^{\mathrm{rev}})\mid (u,v)\in R\}.
\]
A transducer reading from right to left recognizes
$R^{\mathrm{rev}}$, and rational relations are closed under reversal,~\cite[p.~65--66]{Ber79}.

\begin{proposition}
\label{prop:right-mult}
For every column $C$ in~$\col_n(m)$, the relation
\[
K_C=\{(u,v)\in K\times K\mid uC =_{T_n^m} v\}
\]
is rational. Moreover, for all $(u,v)\in K_C$, one has $|v|=|u|+1$.
\end{proposition}

\begin{proof}
Fix $C$ in~$\col_n(m)$.
We define a finite-state transducer that processes the reversed
input word as follows.
Its set of states is~$Q=\col_n(m)\sqcup\{q_f\}$, where the current state in $\col_n(m)$ records the carried column and $q_f$
is a final sink state.
The initial state is $C$, and the unique accepting state is~$q_f$.
For every pair of columns $A,D$ in~$\col_n(m)$ such that~$\col(A\star D)=D'B$,
we add a transition
\[
D \xrightarrow{A\mid B} D'.
\]
Moreover, for every state $D$ in~$\col_n(m)$, we add a final transition~$D \xrightarrow{\varepsilon\mid D} q_f$.
Thus, after the whole input has been read, the machine outputs the last carried
column and moves to the accepting state.

Let~$R_C^{\mathrm{rev}}$ denote the relation recognized by this transducer, and define
\[
R_C
=
\{(u,v)\in \col_n(m)^*\times \col_n(m)^*
\mid
(u^{\mathrm{rev}},v^{\mathrm{rev}})\in R_C^{\mathrm{rev}}\}.
\]
By construction, $R_C^{\mathrm{rev}}$ is rational, and therefore~$R_C$ is rational as well, since rational relations are closed under reversal.

Let~$u=A_1\cdots A_k$ be in~$K$.
Reading~$u$ from right to left, the transducer defines columns
$D_k=C$, $\col(A_i\star D_i)=D_{i-1}B_i$, for~$i=k,k-1,\dots,1$,
and then applies the final transition
\[
D_0 \xrightarrow{\varepsilon\mid D_0} q_f.
\]
Thus the transducer outputs the word~$B_k\cdots B_1D_0$.
After reversal, this yields~$v=D_0B_1\cdots B_k$.
By the invariant established above, the word~$v$ is~$\Colo$-irreducible, hence~$v$ is in~$K$, and it is the~$\Colo$-normal form of~$uC$.
For every~$u\in K$, on input~$u^{\mathrm{rev}}$ the transducer
has the unique accepted output~$v^{\mathrm{rev}}$, where~$v\in K$
is the normal form of~$uC$.
By the definition of~$R_C$, it follows that~$K_C=R_C\cap(K\times K)$.

Since $K$ is regular and rational relations are closed under intersection with regular constraints, it follows that $K_C$ is rational.
Finally, the transducer outputs one column for each input column, together with one extra column via the final $\varepsilon$-transition.
Hence, for every~$(u,v)\in K_C$, one has~$|v|=|u|+1$.
\end{proof}

\subsection{Biautomaticity}
Since the tableaux monoid $T_n^m$ is commutative, for every column~$C$ in~$\col_n(m)$ and every~$u$ in~$K$ we have
\[
Cu =_{T_n^m} uC.
\]
Consequently, the relations describing left and right multiplication
coincide, that is, ${}_C K = K_C$.
By Proposition~\ref{prop:right-mult}, the relation $K_C$ is rational and
satisfies $|v|=|u|+1$ for all $(u,v)$ in~$K_C$.
It follows that~${}_C K$ is also rational and satisfies the same
length bound.

We now combine these results to establish the main automaticity theorem.

\begin{theorem}
\label{thm:biautomatic}
The tableaux monoid $T_n^m$ is biautomatic with respect to the
generating set $\col_n(m)$ and the normal form language~$K$.
\end{theorem}

\begin{proof}
The language $K$ is regular and maps onto $T_n^m$.
Let $C\in\col_n(m)$.
By Proposition~\ref{prop:right-mult}, the relation
\[
K_C=\{(u,v)\in K\times K\mid uC =_{T_n^m} v\}
\]
is rational and satisfies $|v|=|u|+1$.
As observed above, the relation ${}_C K$ is also rational and satisfies
the same length bound.
Hence both $K_C$ and ${}_C K$ are rational relations with length
difference bounded by~$1$.
It follows from Proposition~\ref{prop:rational-to-regular} that the
padded relations~$K_C\delta_R$,~$K_C\delta_L$,~${}_C K\delta_R$, and~${}_C K\delta_L$
are regular.
For the identity multiplier, uniqueness of the normal forms in~$K$
gives
\[
K_\varepsilon={}_\varepsilon K
=\{(u,u)\mid u\in K\}.
\]
Its right-padded and left-padded encodings coincide and are regular,
since~$K$ is regular.

Therefore~$(\col_n(m),K)$ defines a biautomatic structure for $T_n^m$.
\end{proof}

The \emph{word problem} for a monoid asks whether two words over
a generating set represent the same element.
Automatic, and hence biautomatic, monoids have word problem
solvable in quadratic time,~\cite[Corollary~3.7]{CRRT01}.
Since~$T_n^m$ is biautomatic by
Theorem~\ref{thm:biautomatic}, we obtain:

\begin{corollary}
\label{cor:word-problem}
The word problem for the tableaux monoid~$T_n^m$ is solvable
in quadratic time.
\end{corollary}

Given elements~$a,b$ of a monoid~$\M$, we say that~$a$
\emph{left-divides}~$b$ if~$ax=b$ for some~$x\in\M$, and
that~$a$ \emph{right-divides}~$b$ if~$ya=b$ for some~$y\in\M$.
The corresponding decision problems are called the \emph{left-divisibility problem} and the \emph{right-divisibility problem}, respectively.

For the tableaux monoid, these problems can be decided directly
from the definition of the star product.
Since the monoid is commutative, left and right divisibility
coincide.
Let~$T$ and~$U$ be in~$T_n^m$.
The tableau~$T$ divides~$U$ if and only if there exists a tableau
$X$ in~$T_n^m$ such that~$U=T\star X$.
In that case, the number of columns of~$X$ is the number of
columns of~$U$ minus that of~$T$.
If this difference is negative, divisibility is impossible.
Otherwise, there are only finitely many tableaux with this
number of columns, height at most~$n$, and entries in~$[m]$.
These tableaux can be effectively enumerated, and each
candidate~$X$ can be tested for~$T\star X=U$.
Thus we obtain:

\begin{corollary}
The tableaux monoid~$T_n^m$ has decidable left- and
right-divisibility problems.
\end{corollary}

Recall that Green's relations~$\mathcal R$ and~$\mathcal L$ on a monoid~$\M$ are defined by~$a\,\mathcal R\,b$ if and only if~$a\M=b\M$, and~$a\,\mathcal L\,b$ if and only if~$\M a=\M b$. 
The remaining Green relations are defined by~$
a\,\mathcal J\,b$ if and only if~$\M a\M=\M b\M$,~$\mathcal H=\mathcal L\cap\mathcal R$,
and by taking~$\mathcal D$ to be the smallest equivalence relation
containing~$\mathcal L\cup\mathcal R$.
For~$\mathcal G$ in~$
\{\mathcal R,\mathcal L,\mathcal J,\mathcal H,\mathcal D\}$,
a monoid is called \emph{$\mathcal G$-trivial} if
$a\,\mathcal G\,b$ implies~$a=b$.

The definition of the star product also gives the following
structural property, independently of biautomaticity.

\begin{corollary}
\label{cor:green-trivial}
The tableaux monoid~$T_n^m$ is
$\mathcal R$-trivial, $\mathcal L$-trivial, $\mathcal J$-trivial,
$\mathcal H$-trivial, and $\mathcal D$-trivial.
\end{corollary}

\begin{proof}
Suppose that~$T\,\mathcal R\,U$.
Then there exist tableaux~$X,Y$ such that~$U=T\star X$ and~$T=U\star Y$.
The first-row length, taken to be zero for the empty tableau,
is additive under~$\star$.
Applying this observation to the two equalities shows that
the sum of the first-row lengths of~$X$ and~$Y$ is zero.
Hence~$X=Y=\emptyset$, and therefore~$T=U$.
This proves~$\mathcal R$-triviality.

Since the monoid is commutative, the principal left, right,
and two-sided ideals generated by an element coincide.
Consequently,~$\mathcal L=\mathcal R=\mathcal J$ are equality
relations. By their definitions,~$\mathcal H$ and~$\mathcal D$
are equality relations as well.
\end{proof}

\begin{remark}
Deciding any Green relation on~$T_n^m$ is therefore equivalent
to deciding equality. By Corollary~\ref{cor:word-problem},
all Green relations on~$T_n^m$ are decidable in quadratic time.
\end{remark}

\section{Finite coherent presentations of the tableaux monoid}
\label{S:ColumnCoherentPresentationsTableauxMonoid}

We begin by recalling the notion of coherent presentations of monoids
in the framework of polygraphs; see~\cite{GaussentGuiraudMalbos15,GuiraudMalbos18}
for detailed accounts.
We then extend the column presentation to a finite coherent presentation
of the tableaux monoid and study its homological consequences via
higher-dimensional rewriting.
Finally, we describe actions of the tableaux monoid on categories
in terms of this coherent presentation.

\subsection{Coherent presentations and polygraphs}
\label{SS:CoherentPresentations}

Rewriting systems can be interpreted as $2$-polygraphs with a unique $0$-cell.
A \emph{$2$-polygraph with a unique $0$-cell} $P$ consists of a pair $(P_1,P_2)$, where $P_1$ is a set
of generating $1$-cells and~$P_2$ is a globular extension of the free monoid
$P_1^\ast$, viewed as a $1$-category.
The elements of $P_2$ are called \emph{generating $2$-cells} and are written~$\beta : u \;\dfl\; v$, 
where~$u,v\in P_1^\ast$ are respectively the \emph{source} and the
\emph{target} of~$\beta$,  denoted by~$s_1(\beta)$ and~$t_1(\beta)$.
Any rewriting system $R$ on a set $X$ gives rise to such a $2$-polygraph, whose
generating $2$-cells are precisely the rewriting rules of~$R$.
Recall that a \emph{$(2,1)$-category} is a category enriched in groupoids.
We denote by~$\tck{P}_2$ the $(2,1)$-category freely generated by the
$2$-polygraph~$P$.
A pair $(f,g)$ of $2$-cells in~$\tck{P}_2$ satisfying~$s_1(f)=s_1(g)$ and~$t_1(f)=t_1(g)$
is called a \emph{$2$-sphere} of~$\tck{P}_2$.
A \emph{$(3,1)$-polygraph} is a pair~$(P,P_3)$ consisting of a $2$-polygraph
$P$ together with a globular extension $P_3$ of the $(2,1)$-category
$\tck{P}_2$.
Its elements are \emph{generating $3$-cells}
\[
A : f \;\tfl\; g,
\]
where $(f,g)$ is a $2$-sphere in~$\tck{P}_2$.
The $2$-cells $f$ and $g$ are respectively called the \emph{source} and the
\emph{target} of~$A$, and are denoted by $s_2(A)$ and $t_2(A)$.
Such a $3$-cell may be represented by the following globular diagrams:
\[
\xymatrix @C=8em {
\bullet
	\ar @/^4ex/ [r] ^-{u} _-{}="src"
	\ar @/_4ex/ [r] _-{v} ^-{}="tgt"
	\ar@2 "src"!<-15pt,-10pt>;"tgt"!<-15pt,10pt> _-*+{f} ^-{}="srcA"
	\ar@2 "src"!<+15pt,-10pt>;"tgt"!<+15pt,10pt> ^-*+{g} _-{}="tgtA"
	\ar@3 "srcA"!<8pt,0pt> ; "tgtA"!<-8pt,0pt> ^-{A}
&
\bullet
}
\qquad\text{or equivalently}\qquad
\xymatrix@!C@C=4.5em{
u
	\ar@2@/^3ex/ [r] ^{f} _{}="src"
	\ar@2@/_3ex/ [r] _{g} ^{}="tgt"
& 
v
\ar@3 "src"!<0pt,-7pt>;"tgt"!<0pt,7pt> ^-{A}
}
\]
where $\bullet$ denotes the unique $0$-cell of~$P$.
We denote by~$\tck{P}_3$ the free $(3,1)$-category generated by the
$(3,1)$-polygraph $(P,P_3)$.
An \emph{extended presentation} of a monoid $\M$ is a $(3,1)$-polygraph whose
underlying $2$-polygraph presents~$\M$.
It is said to be a \emph{coherent presentation} if the cellular extension~$P_3$ of the $(2,1)$-category $\tck{P}_2$ is \emph{acyclic}, that is, if for every
$2$-sphere $(f,g)$ in~$\tck{P}_2$ there exists a $3$-cell $A$ in~$\tck{P}_3$ with
$s_2(A)=f$ and $t_2(A)=g$.
The $3$-cells of~$\tck{P}_3$ are called the \emph{syzygies} of the presentation.
Squier’s coherence theorem,~\cite{Squier94, GuiraudMalbos18}, asserts that any convergent rewriting system~$R$ presenting a monoid~$\M$
can be extended into a coherent presentation of~$\M$, by adding, for each critical branching $(f,g)$ of~$R$, a
generating $3$-cell of the form
\[
\xymatrix @R=0.8em @C=2.5em @!C {
        & v
        \ar@2@/^/  [dr] ^-{f'}
        \ar @3[]!<0pt,-10pt>;[dd]!<0pt,10pt> ^-*+{A_{f,g}}
        \\
        u
        \ar@2@/^/  [ur] ^-{f}
         \ar@2@/_/ [dr] _-{g}
        && w
        \\
        & v'
         \ar@2@/_/ [ur] _-{g'}
}
\]
where $f'$ and $g'$ are chosen confluent rewriting paths.

\subsection{A finite coherent extension of the column presentation}
We extend the column presentation into a finite coherent presentation of the tableaux monoid by showing that all critical branchings admit uniformly bounded confluence diagrams and by adjoining a finite family of generating $3$-cells.

For~$\Colo$, we choose one representative of each critical branching
up to exchange of its two branches, listing the rewrite of~$AB$
first and the rewrite of~$BC$ second.
All counts below use this convention.

\begin{proposition}
\label{prop:three-steps}
Let $A$, $B$, and $C$ be in~$\col_n(m)$, and consider a critical branching of~$\Colo$ with source~$ABC$.
Then each of the two one-step reducts of~$ABC$ admits a rewriting path of length at most~$2$ to the common normal form~$\col(A\star B\star C)$.
Equivalently, every maximal rewriting sequence starting from~$ABC$ reaches the normal form in at most~$3$ steps.
\end{proposition}

\begin{proof}
Let~$ABC \;\dfl\; A_1B_1C$ and $ABC \;\dfl\; A\widetilde B\widetilde C$ be the two one-step reductions of the critical branching, where~$\col(A\star B)=A_1B_1$ and~$\col(B\star C)=\widetilde B\widetilde C$.
We first consider the branch beginning with the rewrite of~$AB$.
By construction, the pair~$A_1B_1$ is semistandard.
If $B_1C$ is also semistandard, then $A_1B_1C$ is already irreducible, and this branch reaches the normal form in one step.
Assume now that $B_1C$ is not semistandard, and write~$\col(B_1\star C)=B_2C_2$.
Then
\[
A_1B_1C \;\dfl\; A_1B_2C_2,
\]
and the pair~$B_2C_2$ is semistandard by construction.
If $A_1B_2$ is semistandard, then $A_1B_2C_2$ is irreducible, so this branch reaches the normal form in two steps.
Assume finally that $A_1B_2$ is not semistandard, and write~$\col(A_1\star B_2)=A_3B_3$.
Then
\[
A_1B_2C_2 \;\dfl\; A_3B_3C_2,
\]
and the pair~$A_3B_3$ is semistandard.
It remains to show that $B_3C_2$ is also semistandard.
Using the padded notation introduced in the proof of Theorem~\ref{thm:ColumnPresentation}, the semistandardness of~$A_1B_1$ gives~$\widehat{A_1}(i)\le \widehat{B_1}(i)$, for all~$i$.
Moreover,~$\widehat{B_2}(i)=\min\{\widehat{B_1}(i),\widehat{C}(i)\}$ and~$\widehat{C_2}(i)=\max\{\widehat{B_1}(i),\widehat{C}(i)\}$, 
so both $\widehat{A_1}(i)$ and $\widehat{B_2}(i)$ are bounded above by~$\widehat{C_2}(i)$.
Finally,~$\widehat{B_3}(i)=\max\{\widehat{A_1}(i),\widehat{B_2}(i)\}$,
hence
\[
\widehat{B_3}(i)\le \widehat{C_2}(i)
\qquad\text{for all } i.
\]
Thus the pair~$B_3C_2$ is semistandard, and therefore $A_3B_3C_2$ is irreducible.
We conclude that any maximal rewriting sequence beginning with the left rewrite reaches an irreducible word in at most three steps.
The argument for the branch beginning with the rewrite of~$BC$ is symmetric.
Since~$\Colo$ is convergent, the irreducible word obtained from either branch must be the unique normal form~$\col(A\star B\star C)$ of~$ABC$.
\end{proof}

By Proposition~\ref{prop:three-steps}, the confluence diagrams
can be chosen with at most three rewriting steps on each branch.
Since~$\Colo$ is finite and convergent, Squier's coherence theorem
gives the following finite coherent presentation.

\begin{theorem}
\label{T:ColumnCoherentPresentation}
The column $2$-polygraph~$\Colo$ admits a finite coherent extension
$\Colo_3$ as a $(3,1)$-polygraph presenting the tableaux monoid~$T_n^m$.

More precisely, let $A,B,C$ be in~$\col_n(m)$ such that
$AB$ and $BC$ are not semistandard, and write
\[
\col(A\star B)=A_1B_1,\quad \col(B_1\star C)=B_2C_2,\quad \col(B\star C)=B_3C_3,\quad \col(A\star B_3)=A_4B_4.
\]
Then adjoining a generating $3$-cell $\Xr_{A,B,C}$ filling the
corresponding critical branching
\[
\xymatrix @R=0.8em @C=2.5em @!C{
& A_1B_1C
\ar@2 [r] ^-{A_1\alpha_{B_1,C}}
\ar @3[]!<55pt,-10pt>;[dd]!<55pt,10pt> ^-*+{\Xr_{A,B,C}}
& A_1B_2C_2
\ar@2@/^/ [dr] ^-{\alpha_{A_1,B_2}C_2}
\\
ABC
\ar@2@/^/ [ur] ^-{\alpha_{A,B}C}
\ar@2@/_/ [dr] _-{A\alpha_{B,C}}
&&&
\col(A\star B\star C)
\\
& AB_3C_3
\ar@2 [r] _-{\alpha_{A,B_3}C_3}
& A_4B_4C_3
\ar@2@/_/ [ur] _-{A_4\alpha_{B_4,C_3}}
}
\]
for each such triple $(A,B,C)$ yields a finite coherent presentation of
the tableaux monoid~$T_n^m$, with the convention that whenever one of the
pairs $B_1C$, $A_1B_2$, $AB_3$, or $B_4C_3$ is semistandard,
the corresponding arrow in the diagram is the appropriate identity $2$-cell.
\end{theorem}

\subsubsection{Polygraphic resolution and Koszulity}
Endow~$\Bbbk[T_n^m]$ with the connected grading determined by
$\deg(\emptyset)=0$ and~$\deg(T)=\lambda_1$ when~$T$ has shape~$\lambda$. Equivalently, every column generator has degree~$1$. The quadratic relations of~$\Colo$ are homogeneous for this
grading.

We use the convention for polygraphs of associative algebras from~\cite{GuiraudHoffbeckMalbos19}: algebra generators are generating $0$-cells and defining relations are generating $1$-cells. Thus cell dimensions are shifted down by one relative to the monoid polygraphs used above.

For~$N\ge2$, define the function~$\ell_N:\mathbb N\to\mathbb N$ by
\[
\ell_N(2q)=Nq,
\qquad
\ell_N(2q+1)=Nq+1
\qquad(q\ge0).
\]
A graded polygraphic resolution~$X$ is called
\emph{$\ell_N$-concentrated} if every generating $p$-cell of~$X$ is
homogeneous of internal degree~$\ell_N(p+1)$.
For~$N=2$ one has~$\ell_2(q)=q$, so $\ell_2$-concentration means that
the generating $p$-cells have internal degree~$p+1$. Equivalently, the
associated free bimodule in homological degree~$p+1$ is generated in
internal degree~$p+1$.

We use the graded augmentation
$\epsilon_{\mathrm{gr}}:\Bbbk[T_n^m]\to\Bbbk$,
which sends~$\emptyset$ to~$1$ and every nonempty tableau to~$0$.
Recall that a connected graded algebra~$A$ is \emph{Koszul}
if the graded left~$A$-module~$\Bbbk$, defined by its
augmentation, admits a graded free resolution whose module
in homological degree~$r$ is generated in internal degree~$r$.

The linearization of~$\Colo$ is a quadratic convergent presentation of
the graded algebra~$\Bbbk[T_n^m]$.
By~\cite[Theorem~6.2.4]{GuiraudHoffbeckMalbos19}, this reduced
convergent presentation admits Squier's polygraphic
resolution. This resolution is~$\ell_2$-concentrated:
its generating $p$-cells arise from iterated overlaps
supported on~$p+1$ column generators and therefore have
internal degree~$p+1$.
 Hence~\cite[Proposition~7.2.2]{GuiraudHoffbeckMalbos19} applies and yields a
linear graded bimodule resolution. Consequently:

\begin{corollary}
\label{cor:Koszul}
Let~$\Bbbk$ be a field. The tableaux algebra~$\Bbbk[T_n^m]$ is Koszul.
\end{corollary}

This recovers, by higher-dimensional rewriting, the Koszulity established
by PBW methods in
\cite[Corollary~2.14]{DaughertyGonzalezMunizOcalPanTorres25}.

\subsection{Counting generators, rules, and critical branchings}
We now count the generators, rewriting rules, and critical branchings of the
column presentation~$\Colo$.
The idea is to express these numbers in terms of tableaux with two and three
columns.
The following lemma provides explicit formulas for the numbers of such tableaux.

\begin{lemma}
\label{lem:PmnnQmnn}
Let
\[\displaystyle 
P_{m,n}
=
\sum_{1\le l\le k\le n}\#\YoungT_m(2^l,1^{k-l}),
\qquad
Q_{m,n}
=
\sum_{1\le j\le l\le k\le n}\#\YoungT_m(3^j,2^{l-j},1^{k-l}).
\]
Then
\[\displaystyle 
P_{m,n}
=
\sum_{1\le l\le k\le n}
\frac{k-l+1}{k+1}\binom{m}{k}\binom{m+1}{l},
\]
and
\[\displaystyle 
Q_{m,n}
=
\sum_{1\le j\le l\le k\le n}
\frac{(k-l+1)(l-j+1)(k-j+2)}
{(k+1)(k+2)(l+1)}
\binom{m}{k}\binom{m+1}{l}\binom{m+2}{j}.
\]
\end{lemma}

\begin{proof}
We use the hook-content formula for semistandard Young tableaux~\cite[Theorem~7.21.2]{StanleyEC2}:
for every partition~$\lambda$,
\[\displaystyle 
\#\YoungT_m(\lambda)
=
\prod_{x\in\lambda}\frac{m+c(x)}{h(x)},
\]
where, for a cell $x=(r,s)$ in row~$r$ and column~$s$, we have
\[
c(x)=s-r
\quad
\text{and}
\quad
h(x)=\lambda_r-s+\lambda'_s-r+1.
\]
Here~$\lambda_r$ (resp.~$\lambda'_s$) denotes the number of boxes in the~$r$-th row (resp.~$s$-th column) of~$\lambda$.

First, let $\lambda=(2^l,1^{k-l})$.
Its first column has height~$k$ and its second column has height~$l$.
Hence the numerator is
\[\displaystyle 
\prod_{i=1}^{k}(m+1-i)\;\prod_{i=1}^{l}(m+2-i)
=
k!\binom{m}{k}\,l!\binom{m+1}{l}.
\]
For the hook lengths, the cells in the first (resp. second) column contribute
\[\displaystyle 
\prod_{i=1}^{l}(k-i+2)\;\prod_{i=l+1}^{k}(k-i+1),
\qquad
\big(
\text{resp.}~\prod_{i=1}^{l}(l-i+1)
\big).
\]
Thus the total hook product is~$\displaystyle \frac{(k+1)!\,l!}{k-l+1}$.
Therefore
\[\displaystyle 
\#\YoungT_m(2^l,1^{k-l})
=
\frac{k-l+1}{k+1}\binom{m}{k}\binom{m+1}{l}.
\]

Now let $\lambda=(3^j,2^{l-j},1^{k-l})$.
Its columns have heights~$k$, $l$, and~$j$.
The numerator is
\[
\prod_{i=1}^{k}(m+1-i)\;
\prod_{i=1}^{l}(m+2-i)\;
\prod_{i=1}^{j}(m+3-i)
=
k!\binom{m}{k}\,l!\binom{m+1}{l}\,j!\binom{m+2}{j}.
\]
For the hook lengths, the first, second, and third columns contribute
\[
\prod_{i=1}^{j}(k-i+3)\;
\prod_{i=j+1}^{l}(k-i+2)\;
\prod_{i=l+1}^{k}(k-i+1),
\quad
\prod_{i=1}^{j}(l-i+2)\;
\prod_{i=j+1}^{l}(l-i+1),
\quad
\text{ and }
\quad
\prod_{i=1}^{j}(j-i+1).
\]
Multiplying these factors gives
\[
\frac{(k+2)!\,(l+1)!\,j!}
{(k-l+1)(l-j+1)(k-j+2)}.
\]
Hence
\[
\#\YoungT_m(3^j,2^{l-j},1^{k-l})
=
\frac{(k-l+1)(l-j+1)(k-j+2)}
{(k+1)(k+2)(l+1)}
\binom{m}{k}\binom{m+1}{l}\binom{m+2}{j}.
\]
Summing over all admissible indices yields the stated formulas for
$P_{m,n}$ and~$Q_{m,n}$.
\end{proof}

\begin{proposition}
\label{prop:counting-final}
Let $m$ and $n$ be in~$\mathbb N$ with $m\ge n$.
\begin{enumerate}[\bf i)]
\item The number of generators of~$\Colo$ is
\[
N_{m,n}= |\col_n(m)|=\sum_{k=1}^{n}\binom{m}{k}.
\]

\item The number of rewriting rules of~$\Colo$ is
\[
|\Colo_2|=N_{m,n}^2-P_{m,n}.
\]

\item The number of critical branchings of~$\Colo$ is
\[
|\mathrm{Crit}(\Colo)|
=
N_{m,n}^3-2N_{m,n}P_{m,n}+Q_{m,n}.
\]
\end{enumerate}
\end{proposition}

\begin{proof}
A column of height~$k$ is a strictly increasing $k$-tuple in~$[m]$, and hence
corresponds to a $k$-subset of~$[m]$.
There are $\displaystyle \binom{m}{k}$ such columns, and summing over $1\le k\le n$ proves~{\bf i)}.

For~{\bf ii)}, the rules of~$\Colo$ correspond to ordered pairs
$(A,B)\in\col_n(m)^2$ such that~$AB$ is non-semistandard.
Thus $|\Colo_2|$ is the total number of ordered pairs of columns minus the number of semistandard ones.
Fix $k,l$ with $1\le l\le k\le n$.
A pair~$AB$ with~$\height(A)=k$ and~$\height(B)=l$ is semistandard if and only if, when $B$ is placed to the right of~$A$, we obtain a tableau of shape~$(2^l,1^{k-l})$.
Therefore the number of semistandard pairs of heights $(k,l)$ is exactly~$\displaystyle \#\YoungT_m(2^l,1^{k-l})$.
Summing over all possible heights gives~$P_{m,n}$, and thus~$|\Colo_2|=N_{m,n}^2-P_{m,n}$.

For~{\bf iii)}, since~$\Colo$ is quadratic and reduced, its critical branchings are exactly the triples~$ABC$ of columns such that both adjacent pairs~$AB$ and~$BC$ are non-semistandard.
We count such triples by inclusion--exclusion.
There are $N_{m,n}^3$ triples of columns in total.
The number of triples for which~$AB$ is semistandard is~$N_{m,n}P_{m,n}$, since one chooses a semistandard pair~$AB$ and then an arbitrary column~$C$.
Similarly, the number of triples for which~$BC$ is semistandard is again~$N_{m,n}P_{m,n}$.
It remains to count the triples for which both~$AB$ and~$BC$ are semistandard.
Fix~$k,l,j$ with $1\le j\le l\le k\le n$.
A triple~$ABC$ with $\height(A)=k$, $\height(B)=l$, and $\height(C)=j$
satisfies that both~$AB$ and~$BC$ are semistandard if and only if the three columns together form a tableau of shape~$(3^j,2^{l-j},1^{k-l})$.
Hence the number of such triples is~$\displaystyle \#\YoungT_m(3^j,2^{l-j},1^{k-l})$,
and summing over all possible heights gives~$Q_{m,n}$.
Therefore
\[|\mathrm{Crit}(\Colo)|
=
N_{m,n}^3-2N_{m,n}P_{m,n}+Q_{m,n}.
\]
\end{proof}

\subsection{Example: the tableaux monoid for $m=3$ and $n=2$}
We illustrate the column presentation and its critical branchings
in a small representative case.
The columns of height at most $2$ are
\[
(1),\ (2),\ (3),\ (1,2),\ (1,3),\ (2,3),
\]
which form the generating set $\col_2(3)$ of the tableaux monoid $T_2^3$. 
In this case,~$N_{3,2}=6$,~$P_{3,2}=20$ and~$Q_{3,2}=50$,
so Proposition~\ref{prop:counting-final} yields~$|\Colo_2|=16$ and~$|\mathrm{Crit}(\Colo)|=26$.

The column presentation $\Colo$ consists of the following $16$ quadratic rules:
\begin{enumerate}[\bf i)]
\item \emph{Letter-letter rules}:~$(2)(1)\dfl (1)(2)$,~$(3)(1)\dfl (1)(3)$ and~$(3)(2)\dfl (2)(3)$.
\item \emph{Letter-column rules}:
\[
\begin{array}{lllll}
(1)(1,2)\dfl(1,2)(1), && (1)(1,3)\dfl (1,3)(1),& &(2)(1,2)\dfl(1,2)(2),\\
(2)(1,3)\dfl (1,3)(2), && (2)(2,3)\dfl (2,3)(2), &&(3)(1,2)\dfl(1,2)(3),  \\
(3)(1,3)\dfl (1,3)(3), && (3)(2,3)\dfl (2,3)(3).
\end{array}
\]
\item \emph{Column-column rules}:
\[
(1,3)(1,2)\dfl (1,2)(1,3),\qquad
(2,3)(1,2)\dfl (1,2)(2,3),\qquad
(2,3)(1,3)\dfl (1,3)(2,3).
\]
\item \emph{Interaction rules}:~$(2,3)(1)\dfl(1,3)(2)$ and~$(1)(2,3)\dfl (1,3)(2)$.
\end{enumerate}
The $26$ critical branchings decompose into the following families:
\begin{enumerate}[\bf i)]
\item One braid-type branching coming from the decreasing triple of letters:~$(3)(2)(1)$. In this case, the corresponding generating $3$-cell is the following:
\[
\xymatrix @R=0.8em @C=2.5em @!C{
& (2)(3)(1)
\ar@2 [r] ^-{(2)\alpha_{(3),(1)}}
\ar @3[]!<15pt,-10pt>;[dd]!<15pt,10pt> ^-*+{\Xr_{(3),(2),(1)}}
& (2)(1)(3)
\ar@2@/^/ [dr] ^-{\alpha_{(2),(1)}(3)}
\\
(3)(2)(1)
\ar@2@/^/ [ur] ^-{\alpha_{(3),(2)}(1)}
\ar@2@/_/ [dr] _-{(3)\alpha_{(2),(1)}}
&&&
(1)(2)(3)
\\
& (3)(1)(2)
\ar@2 [r] _-{\alpha_{(3),(1)}(2)}
& (1)(3)(2)
\ar@2@/_/ [ur] _-{(1)\alpha_{(3),(2)}}
}
\]
\item Nine branchings of the form~$xyA$,
where
\[
xy\in\{(2)(1),(3)(1),(3)(2)\}
\quad\text{and}\quad
A\in\{(1,2),(1,3),(2,3)\}.
\]
\item Three branchings of the form~$x(2,3)(1)$, where~$x\in\{(1),(2),(3)\}$.
\item Nine branchings of the form~$xAB$,
where
\[
x\in\{(1),(2),(3)\}
\quad\text{and}\quad
AB\in\{(1,3)(1,2),\ (2,3)(1,2),\ (2,3)(1,3)\}.
\]
\item One braid-type branching coming from the decreasing triple of columns:~$(2,3)(1,3)(1,2)$.
\item Three branchings of the form~$(2,3)(1)A$, where~$A\in\{(1,2),(1,3),(2,3)\}$.
\end{enumerate}
The remaining generating $3$-cells are obtained by specializing the generic confluence diagram of Theorem~\ref{T:ColumnCoherentPresentation} to each of the critical branchings listed above.
For brevity, we display only the braid-type branching on letters.

\subsection{Actions of the tableaux monoid on categories}
\label{SS:Actions}

We briefly indicate how the coherent column presentation can be used to
describe categorical actions of the tableaux monoid.

Recall from~\cite{GaussentGuiraudMalbos15} that an \emph{action of a monoid $\M$ on a category $\mathbf{C}$} consists of an endofunctor $A(u):\mathbf{C}\to\mathbf{C}$ for each element~$u$ in~$\M$, together with natural isomorphisms 
\[ 
A_{u,v}:A(v)\circ A(u)\Rightarrow A(uv), 
\] 
and a unit isomorphism
\[
A_0:\operatorname{Id}_{\mathbf C}\Rightarrow A(1),
\]
subject to the usual associativity and unit coherence conditions. Denote by~$\Act(\M)$ the category whose objects are actions
of~$\M$ on categories and whose morphisms are pseudonatural
transformations.
Let~$\Sigma=(P,P_3)$ be a coherent presentation of a monoid~$\M$, where
$P$ is its underlying $2$-polygraph. Denote by~$\overline{\Sigma}=\tck{P}_2/P_3$
the associated $(2,1)$-category, obtained by identifying the source and
target of every generating $3$-cell. By
\cite[Theorem~5.1.6]{GaussentGuiraudMalbos15}, there is an equivalence
of categories
\[
\Act(\M)\simeq 2\Cat(\overline{\Sigma},\Cat).
\]
Here $2\Cat(\overline{\Sigma},\Cat)$ denotes the category of strict $2$-functors and pseudonatural transformations.
Equivalently, an object on the right may be described as a $2$-functor
from~$\tck{P}_2$ to~$\Cat$ that takes the two boundary $2$-cells of every
element of~$P_3$ to equal natural transformations.

Applying this correspondence to the coherent column presentation~$\Colo_3$ of~$T_n^m$ yields the following explicit description of actions of the tableaux monoid on categories.

\begin{theorem}
\label{thm:actions}
Up to equivalence, an action of the tableaux monoid~$T_n^m$
on a category~$\mathbf C$ is described by the following data:
\begin{enumerate}[\bf i)]
\item an endofunctor $F_A:\mathbf C\to\mathbf C$ for each column
$A$ in~$\col_n(m)$;
\item for every rewriting rule
\[
\alpha_{A,B}:AB \;\dfl\; A'B'
\]
of~$\Colo$, a natural isomorphism
\[
\alpha_{A,B}^\mathbf{C}:
F_B\circ F_A
\Rightarrow
F_{B'}\circ F_{A'};
\]
\item for every generating $3$-cell~$\Xr_{A,B,C}$ of~$\Colo_3$, the
two composites of natural isomorphisms corresponding to its source and
target are equal. Equivalently, the associated coherence diagram in~$\Cat$ commutes.
\end{enumerate}
\end{theorem}

\begin{example}
Consider the tableaux monoid $T_2^3$.
One of the rewriting rules of its column presentation is
\[
(2,3)(1)\;\dfl\;(1,3)(2).
\]
Let $\mathbf C$ be a category.
An action of $T_2^3$ on $\mathbf C$ therefore consists of choosing
endofunctors
\[
F_{(1)},\ F_{(2)},\ F_{(3)},\
F_{(1,2)},\ F_{(1,3)},\ F_{(2,3)} :
\mathbf C \to \mathbf C,
\]
together with a natural isomorphism
\[
\alpha_{(2,3),(1)}:
F_{(1)}\circ F_{(2,3)}
\Rightarrow
F_{(2)}\circ F_{(1,3)}.
\]
Each rewriting rule of the column presentation similarly yields a
natural isomorphism between the corresponding composites of
endofunctors.
The coherent presentation further imposes compatibility relations
between these isomorphisms.
For instance, since $x(2,3)(1)$ is a critical branching for each
$x$ in~$\{(1),(2),(3)\}$, there are two distinct rewriting sequences
starting from the triple $x(2,3)(1)$.
When $x=(2)$, these two sequences produce two composites of natural
isomorphisms from
\[
F_{(1)}\circ F_{(2,3)}\circ F_{(2)}
\]
to the functor corresponding to the normal form of $(2)(2,3)(1)$.
The generating $3$-cell $\Xr_{(2),(2,3),(1)}$ ensures that these two
composites coincide.
Thus an action of $T_2^3$ is described by explicit endofunctors together
with natural isomorphisms satisfying the required coherence diagrams.
\end{example}

\section*{Acknowledgements}
We thank the referee for a careful reading of the manuscript
and for helpful comments and suggestions that improved
the clarity and presentation of the paper.

\begin{small}
\renewcommand{\refname}{\Large\textsc{References}}
\bibliographystyle{plain}
\bibliography{biblioCURRENT}

\def\cprime{$'$}
\begin{thebibliography}{10}

\bibitem{Per2016}
Per Alexandersson.
\newblock Gelfand–tsetlin polytopes and the integer decomposition property.
\newblock {\em European Journal of Combinatorics}, 54:1--20, 2016.

\bibitem{Ber79}
Jean Berstel.
\newblock {\em {Transductions and context-free languages}}, volume~38 of {\em
  Leitf{\"a}den der Angewandten Mathematik und Mechanik [Guides to Applied
  Mathematics and Mechanics]}.
\newblock {B. G. Teubner}, 1979.

\bibitem{BookOtto93}
Ronald Book and Friedrich Otto.
\newblock {\em String-rewriting systems}.
\newblock Texts and Monographs in Computer Science. Springer-Verlag, 1993.

\bibitem{CainGrayMalheiro15}
Alan~J. Cain, Robert~D. Gray, and Ant{\'o}nio Malheiro.
\newblock Finite {G}r\"obner--{S}hirshov bases for {P}lactic algebras and
  biautomatic structures for {P}lactic monoids.
\newblock {\em J. Algebra}, 423:37--53, 2015.

\bibitem{CainGrayMalheiro15b}
Alan~J. Cain, Robert~D. Gray, and Ant{\'o}nio Malheiro.
\newblock Rewriting systems and biautomatic structures for {C}hinese,
  hypoplactic, and {S}ylvester monoids.
\newblock {\em Internat. J. Algebra Comput.}, 25(1-2):51--80, 2015.

\bibitem{CainGrayMalheiro19}
Alan~J. Cain, Robert~D. Gray, and Ant\'{o}nio Malheiro.
\newblock Crystal monoids \& crystal bases: rewriting systems and biautomatic
  structures for plactic monoids of types {$A_n$}, {$B_n$}, {$C_n$}, {$D_n$},
  and {$G_2$}.
\newblock {\em J. Combin. Theory Ser. A}, 162:406--466, 2019.

\bibitem{CRRT01}
Colin~M. Campbell, Edmund~F. Robertson, Nikola Ruškuc, and Richard~M. Thomas.
\newblock Automatic semigroups.
\newblock {\em Theoretical Computer Science}, 250(1--2):365--391, 2001.

\bibitem{ChangDuanFraserLi20}
Wen Chang, Bing Duan, Chris Fraser, and Jian-Rong Li.
\newblock Quantum affine algebras and {G}rassmannians.
\newblock {\em Mathematische Zeitschrift}, 296(3-4):1539--1583, 2020.

\bibitem{CurienDuricGuiraud23}
Pierre-Louis Curien, Alen {\DH}uri{\'c}, and Yves Guiraud.
\newblock Coherent presentations of monoids with a right-noetherian garside
  family.
\newblock {\em Journal of Homotopy and Related Structures}, 18:115--152, 2023.

\bibitem{DaughertyGonzalezMunizOcalPanTorres25}
Spencer Daugherty, Nicolle González, Bárbara Muniz, Pablo~S. Ocal, Jianping
  Pan, and Jacinta Torres.
\newblock Structure and geometry of the tableaux algebra.
\newblock arXiv: 2510.27209, 2025.

\bibitem{DuncanRobertsonRuskuc99}
A.~J. Duncan, E.~F. Robertson, and N.~Ru\v{s}kuc.
\newblock Automatic monoids and change of generators.
\newblock {\em Mathematical Proceedings of the Cambridge Philosophical
  Society}, 127(3):403--409, 1999.

\bibitem{FS93}
Christiane Frougny and Jacques Sakarovitch.
\newblock Synchronized rational relations of finite and infinite words.
\newblock {\em Theoretical Computer Science}, 108(1):45--82, 1993.

\bibitem{Fulton97}
William Fulton.
\newblock {\em Young tableaux}, volume~35 of {\em London Mathematical Society
  Student Texts}.
\newblock Cambridge University Press, Cambridge, 1997.
\newblock With applications to representation theory and geometry.

\bibitem{GaussentGuiraudMalbos15}
St{\'e}phane Gaussent, Yves Guiraud, and Philippe Malbos.
\newblock Coherent presentations of {A}rtin monoids.
\newblock {\em Compos. Math.}, 151(5):957--998, 2015.

\bibitem{GuiraudHoffbeckMalbos19}
Yves Guiraud, Eric Hoffbeck, and Philippe Malbos.
\newblock Convergent presentations and polygraphic resolutions of associative
  algebras.
\newblock {\em Math. Z.}, 293(1-2):113--179, 2019.

\bibitem{GuiraudMalbos09}
Yves Guiraud and Philippe Malbos.
\newblock Higher-dimensional categories with finite derivation type.
\newblock {\em Theory Appl. Categ.}, 22:No. 18, 420--478, 2009.

\bibitem{GuiraudMalbos12advances}
Yves Guiraud and Philippe Malbos.
\newblock Higher-dimensional normalisation strategies for acyclicity.
\newblock {\em Adv. Math.}, 231(3-4):2294--2351, 2012.

\bibitem{GuiraudMalbos18}
Yves Guiraud and Philippe Malbos.
\newblock Polygraphs of finite derivation type.
\newblock {\em Math. Structures Comput. Sci.}, 28(2):155--201, 2018.

\bibitem{Hage26}
Nohra Hage.
\newblock Super {Cauchy} identities over signed alphabets.
\newblock To appear in \emph{Lie Theory and Its Applications in Physics},
  https://hal.science/view/index/docid/5424110, 2025.

\bibitem{Hage15}
Nohra Hage.
\newblock Finite convergent presentation of plactic monoid for type {C}.
\newblock {\em Internat. J. Algebra Comput.}, 25(8):1239--1263, 2015.

\bibitem{Hage2021Super}
Nohra Hage.
\newblock Super jeu de taquin and combinatorics of super tableaux of type {A}.
\newblock {\em International Journal of Algebra and Computation},
  32(05):929--952, 2022.

\bibitem{Hage25}
Nohra Hage.
\newblock A super {L}ittlewood--{R}ichardson type rule.
\newblock In Vladimir Dobrev, editor, {\em Lie Theory and Its Applications in
  Physics}, pages 253--259, Singapore, 2025. Springer Nature Singapore.

\bibitem{hage2026coherent}
Nohra Hage.
\newblock Coherent presentations for stylic monoids via row rewriting.
\newblock Preprint arXiv:2609.15515, 2026.

\bibitem{Hage2022RSK}
Nohra Hage.
\newblock A super {R}obinson--{S}chensted--{K}nuth correspondence with symmetry
  and the super {L}ittlewood--{R}ichardson rule.
\newblock {\em Journal of Algebraic Combinatorics}, 63:55, 2026.

\bibitem{HageMalbos17}
Nohra Hage and Philippe Malbos.
\newblock Knuth's coherent presentations of plactic monoids of type {A}.
\newblock {\em Algebr. Represent. Theory}, 20(5):1259--1288, 2017.

\bibitem{HageMalbos22}
Nohra Hage and Philippe Malbos.
\newblock {C}hinese syzygies by insertions.
\newblock {\em Semigroup Forum}, 104(1):88--108, 2022.

\bibitem{HT05}
Michael Hoffmann and Richard~M. Thomas.
\newblock Biautomatic semigroups.
\newblock In M.~Liskiewicz and R.~Reischuk, editors, {\em Fundamentals of
  Computation Theory}, volume 3623 of {\em Lecture Notes in Computer Science},
  pages 56--67. Springer, 2005.

\bibitem{LiDualCanonical26}
Jian-Rong Li.
\newblock Dual canonical bases for unipotent groups and base affine spaces.
\newblock {\em Journal of Algebra and Its Applications}, 25(3):2550354, 2026.

\bibitem{Meha2021}
Uran Meha.
\newblock Coherence for plactic monoids via rewriting theory and crystal
  structures.
\newblock {\em Algebras and Representation Theory}, 26:2289--2312, 2023.

\bibitem{Newman42}
Maxwell Newman.
\newblock On theories with a combinatorial definition of ``equivalence''.
\newblock {\em Ann. of Math. (2)}, 43(2):223--243, 1942.

\bibitem{Squier87}
Craig~C. Squier.
\newblock Word problems and a homological finiteness condition for monoids.
\newblock {\em J. Pure Appl. Algebra}, 49(1-2):201--217, 1987.

\bibitem{Squier94}
Craig~C. Squier, Friedrich Otto, and Yuji Kobayashi.
\newblock A finiteness condition for rewriting systems.
\newblock {\em Theoret. Comput. Sci.}, 131(2):271--294, 1994.

\bibitem{StanleyEC2}
Richard~P. Stanley.
\newblock {\em Enumerative Combinatorics. Volume 2}, volume~62 of {\em
  Cambridge Studies in Advanced Mathematics}.
\newblock Cambridge University Press, 1999.

\end{thebibliography}
\end{small}

\quad

\vfill
\begin{flushright}
\begin{small}
\noindent \textsc{Nohra Hage} \\
\url{nohra.hage@univ-catholille.fr} \\
ICL, Junia, Université Catholique de Lille,\\
 LITL, F-59000 Lille, France\\
\end{small}
\end{flushright}

\vspace{0.25cm}

\begin{small}---\;\;\today\;\;-\;\;\hhmm\;\;---\end{small} \hfill
\end{document}